\documentclass[reqno,11pt]{amsart}

\usepackage{epsfig}
\usepackage{amsmath}
\usepackage{amssymb}
\usepackage{amsthm}
\usepackage{color}
\usepackage{xcolor}
\usepackage{mathrsfs}
\usepackage{tikz}
\usetikzlibrary{math} 

\allowdisplaybreaks

\theoremstyle{plain}

\newtheorem{theorem}{Theorem}[section]

\newtheorem{definition}[theorem]{Definition}

\numberwithin{equation}{section}

\begin{document}

\title[R$_{II}$ type three term relations for bivariate polynomials]{R$_{II}$ type three term relations for bivariate polynomials orthogonal with respect to varying weights}

\author[C. F.~Bracciali, A. M.~Delgado, L.~Fern\'{a}ndez, and T. E.~P\'{e}rez]
{Cleonice F.~Bracciali, Antonia M.~Delgado, Lidia Fern\'{a}ndez, and Teresa E.~P\'{e}rez}

\address[C. F.~Bracciali]{Departamento de Matem\'{a}tica, IBILCE, UNESP - Universidade Estadual Paulista,
Campus at S\~ao Jos\'e do Rio Preto, SP (Brazil)}
\email{cleonice.bracciali@unesp.br}

\address[A. M. Delgado]{Instituto de Matem\'aticas IMAG \&
Departamento de Ma\-te\-m\'{a}\-ti\-ca Aplicada, Facultad de Ciencias. Universidad de Granada (Spain)}
\email{amdelgado@ugr.es}

\address[L. Fern\'andez]{Instituto de Matem\'aticas IMAG \&
Departamento de Ma\-te\-m\'{a}\-ti\-ca Aplicada, Facultad de Ciencias. Universidad de Granada (Spain)}
\email{lidiafr@ugr.es}

\address[T. E. P\'erez]{Instituto
de Matem\'aticas IMAG \&
Departamento de Ma\-te\-m\'{a}\-ti\-ca Aplicada, Facultad de Ciencias. Universidad de Granada (Spain)}
\email{tperez@ugr.es}

\date{\today}

\begin{abstract}
Given a bivariate weight function defined on the positive quadrant of $\mathbb{R}^2$, we study polynomials in two variables orthogonal with respect to varying measures obtained by special modifications of this weight function. In particular, the varying weight functions are given by the multiplication of $x_1^{-n}x_2^{-n}$ times the original weight function. Apart from the question of the existence and construction of such kind of orthogonal polynomials, we show that the systems of bivariate polynomials orthogonal with respect to this kind of varying weights  satisfy R$_{II}$ type three term relations, one for every variable.
A method to construct bivariate orthogonal systems with respect to varying weights based in the Koornwinder's method is developed. Finally, several examples and particular cases have been analysed.
\end{abstract}



\maketitle

\section{Introduction}

Consider a positive measure $\psi(x)$ defined on the real interval $(a,b)$ and the sequence of monic orthogonal polynomials, $\{ P_n (x) \}_{n\geqslant0}$ that satisfy the orthogonality conditions
$$
\int_{a}^{b} x^k P_{n}(x) d \psi(x) = 0,
 \quad 0 \leqslant k \leqslant n-1.
$$
It is well known that the polynomials, $P_n (x),$ satisfy  the three term recurrence relation
$$
P_{n+1}(x) = (x-c_{n+1}) P_{n}(x) - \lambda_{n+1} \, P_{n-1}(x), \quad n \geqslant 0,
$$
with $P_{-1}(x)=0,$  $P_{0}(x)=1$, $\lambda_{n+1}\neq 0$, and that these polynomials are connected with continued fractions known as $J$-fractions, see Chihara \cite{Ch78}.

\bigskip

Within the scope of the connection between orthogonal polynomials and continued fractions Ismail and Masson in \cite{IM95} investigated what they call continued fractions of type R$_{I}$  and of type R$_{II}$.
The continued fractions of type R$_{I}$ are associated with sequences of polynomials that satisfy the R$_{I}$ type three term recurrence relation in the form
$$
P_{n+1}(x) = (x-c_{n+1}) P_{n}(x) - \lambda_{n+1} \, (x-a_{n+1})\, P_{n-1}(x), \quad n \geqslant 0,
$$
with $P_{-1}(x)=0,$  $P_{0}(x)=1,$ and
$\lambda_{n+1} \neq 0.$ If $ P_n(a_{k}) \neq 0,$ for $n \geqslant 1,$ $k=1,2,\ldots, n,$ then there is a linear functional $\mathcal{L}$ such that the rational functions $P_n(x)/\prod_{j=1}^{n}(x-a_k)$ satisfy the orthogonality conditions
$$
\mathcal{L} \left[ x^k \frac{P_n(x)}{\prod_{j=1}^{n}(x-a_k)} \right] = 0, \quad 0 \leqslant k \leqslant n-1.
$$

The special case of R$_{I}$ type three term recurrence relations when $a_{n+1}=0$ appeared in \cite{Ra94}, associated with polynomials that satisfy
\begin{equation} \label{OR_I}
\int_{a}^{b} x^{k} P_{n}(x) \frac{d \psi(x)}{x^n} = 0,
\quad 0 \leqslant k \leqslant n-1,
\end{equation}
where $0 \leqslant a < b \leqslant \infty$ and $\psi(x)$ is a positive measure such that
$$
\int_{a}^{b} x^{k} d \psi(x)  < \infty, \quad k= 0, \pm 1, \pm 2, \dots.
$$
Similar polynomials have first appeared in \cite{JTW80}, related to orthogonal Laurent polynomials (see also \cite{HR86}).
In \cite{SR05}, another special case of R$_{I}$ type three term recurrence relation have been studied with applications to orthogonal polynomials on the unit circle.

\bigskip

On the other hand, continued fractions of type R$_{II}$ are associated with sequences of polynomials that satisfy the R$_{II}$ type three term recurrence relation
$$
P_{n+1}(x) = (x-c_{n+1}) P_{n}(x) - \lambda_{n+1} \, (x-a_{n+1})(x-b_{n+1})\, P_{n-1}(x), \quad n \geqslant 0,
$$
with $P_{-1}(x)=0,$  $P_{0}(x)=1,$ and
$\lambda_{n+1} \neq 0.$ If $P_n(a_{k}) \neq 0$ and $P_n(b_{k}) \neq 0,$ for $n \geqslant 1$ and $k=1,2,\ldots, n,$ there exists a linear functional $\mathcal{L}$ such that the rational functions $P_n(x)/\prod_{j=1}^{n}(x-a_k)(x-b_k)$ satisfy the orthogonality   conditions
$$
\mathcal{L} \left[ x^k \frac{P_n(x)}{\prod_{j=1}^{n}(x-a_k)(x-b_k)} \right] = 0, \quad 0 \leqslant k \leqslant n-1.
$$
Zhedanov in \cite{Zh99} studied the connection of these sequences of polynomials with biorthogonal rational functions and
the generalized eigenvalue problem.
The special case of R$_{II}$ type three term recurrence relation
$$
P_{n+1}(x) = (x-c_{n+1}) P_{n}(x) - \lambda_{n+1} \, (x^2+1) \, P_{n-1}(x), \quad n \geqslant 0,
$$
was studied by Ismail and Sri Ranga \cite{IR19} in connection with orthogonal polynomials on the unit circle. In \cite{IR19}  the authors have given some conditions for coefficients $c_{n+1}$ and $\lambda_{n+1}$ such that these polynomials satisfy
\begin{equation} \label{OR_II}
\int_{-\infty}^{\infty} x^{k} P_{n}(x) \frac{d \psi(x)}{(x^2+1)^n} = 0,  \quad 0 \leqslant k \leqslant n-1,
\end{equation}
where $\psi(x)$ is a positive measure.
\bigskip

Polynomials that satisfy \eqref{OR_I} or \eqref{OR_II} can be
seen as orthogonal polynomials with respect to varying measures. In the literature we can find many references on orthogonal polynomials with respect to varying measures, see \cite{BDMS99, CL98, To98}, and specially on varying measures defined on the unit circle, see \cite{ABMV05, Lo89, Pa94}.

\bigskip

The theory of orthogonal polynomials in one variable has been extended to several variables with many applications (see, for instance, \cite{DX14} and the references therein).
Hence, it seems reasonable to consider if it is possible to extend the R$_{I}$ or the R$_{II}$ type three term relation to several variables. This paper addresses the second relation for the bivariate case, see Theorem \ref{Theorem_TTRR}.

\bigskip

The goal of this paper is to extend results and properties about R$_{II}$ type three term relations for systems of bivariate polynomials that satisfy orthogonality conditions with respect to varying weights in two variables. In particular, we work with varying weight functions of two variables in the form
$$
W_n(x_1,x_2) = x_1^{-n} x_2^{-n}  W(x_1,x_2), \quad n  \geqslant 0,
$$
where $W(x_1,x_2)$ is a weight function defined on a domain $\Omega$ contained on the first quadrant of $\mathbb{R}^2$.
We define and prove the existence of bivariate orthogonal polynomials systems with respect to a varying weight. As main result, we found  R$_{II}$ type three term relations satisfied by these polynomials.
We provide methods to construct the polynomials: one method is based on the moment matrices, another method is inspired by the well-known Koornwinder’s method to construct bivariate orthogonal polynomials. Examples are also presented.

\bigskip

The paper is structured as follows. Section \ref{sec_tools}
brings the basic facts about bivariate polynomial systems.
In Section \ref{sec_R_II} we define varying weight functions of two variables and bivariate orthogonal systems with respect to varying weights. Besides, we prove the existence of these polynomials and how to construct them.
In Section \ref{sec_R_II_TTR} we obtain the R$_{II}$ type three term relations satisfied by the polynomials, as well as 
R$_{II}$ type three term recurrence relations.
In Section \ref{sec_Koorwinder} a method to construct bivariate orthogonal systems with respect to varying weights based in the Koornwinder's method is developed.
The last section brings several examples, in the first one the varying weight function in two variables is given by a product of one variable ones. In this case, the
associated polynomials are given by product of orthogonal polynomials with respect to the varying weight functions in one variable.
Second example uses the construction method developed on Section \ref{sec_Koorwinder}, where the domain $\Omega$ is a triangle  on the first quadrant of $\mathbb{R}^2$.
In the last example we analyse a varying weight function defined in a shifted simplex on $\mathbb{R}^2$.

\section{Basic tools and first results}
\label{sec_tools}

We collect the basic tools and first results that we will need along this work. For more information, see \cite{DX14}.

A bivariate polynomial of total degree $n\geqslant 0$ with real coefficients is a linear combination of the monomials in two variables as
$$
p(x_1,x_2) = \sum_{i+j=0\atop i,j \geqslant 0}^n \,c_{i,j}\,x_1^i\,x_2^j, \qquad c_{i,j}\in \mathbb{R}.
$$
When there exists only one term of highest degree in the form
$$
p(x_1,x_2) = c_{n-k,k}\,x_1^{n-k}\,x_2^k + \sum_{i+j=0\atop i,j\geqslant0}^{n-1} \,c_{i,j}\,x_1^i\,x_2^j,
$$
we say that it is \emph{monomial}, and when $c_{n-k,k}=1$, the polynomial is called \emph{monic}.

Let us denote by $\Pi_n$ the linear space of bivariate polynomials in two variables of total degree less than or equal to $n$ with real coefficients. Then
$$
\Pi_n = \mathrm{span}\{1, x_1, x_2, x_1^2, x_1\,x_2,x_2^2, \ldots, x_1^n, x_1^{n-1}x_2, \ldots, x_2^n\},
$$
and 
\begin{equation}\label{def.tn}
	\dim \Pi_n = (n+2)(n+1)/2=t_n. 
\end{equation}	
Let $\Pi = \cup_{n\geqslant0}\Pi_n.$ 	
The canonical basis for $\Pi$ can be written as a sequence $\{\mathbb{X}_n\}_{n\geqslant0}$ of vectors of increasing size $n+1$ such that $\mathbb{X}_n$ contains the $n+1$ different monomials of total degree $n$ arranged in reverse lexicographical order, this is
\begin{equation}\label{def_Xpos}
\mathbb{X}_n = \begin{pmatrix}
x_1^n\\
x_1^{n-1}\, x_2\\
\vdots\\
x_1\, x_2^{n-1}\\
x_2^n
\end{pmatrix}, \quad n \geqslant 0.
\end{equation}
Observe that the set of entries of $\{\mathbb{X}_m\}_{m=0}^n$ forms a basis of $\Pi_n$.

Now, we define the vectors for negative indexes.

\begin{definition} For $n\geqslant1$, we define
\begin{equation*}
\mathbb{X}_{-n}  = x_1^{-n}x_2^{-n} \mathbb{X}_{n} = \begin{pmatrix}
	x_2^{-n}\\
	x_1^{-1}\, x_2^{-n+1}\\
	\vdots\\
	x_1^{-n+1}\, x_2^{-1}\\
	x_1^{-n}
\end{pmatrix}.
\end{equation*}
\end{definition}

Consider $L_{n,1}$ and $L_{n,2}$ matrices of size $(n+1)\times (n+2)$  given by
\begin{equation}\label{L_ni}
L_{n,1} = \left(\begin{array}{ccc|c}
1        &        & \bigcirc & 0     \\
         & \ddots &          & \vdots \\
\bigcirc &        & 1        & 0
\end{array}\right)
\quad \mbox{and} \quad
L_{n,2} = \left(\begin{array}{c|ccc}
0      & 1        &        & \bigcirc \\
\vdots &          & \ddots &          \\
0      & \bigcirc &        & 1
\end{array}\right).
\end{equation}
Observe that $L_{n,i}$ are full rank matrices,
such that $L_{n,i}\,L_{n,i}^T = I_{n+1}$. For $n \geqslant 0$, we get
\begin{equation*}
x_1\,\mathbb{X}_n = x_1\,\begin{pmatrix}
x_1^n\\
x_1^{n-1}\, x_2\\
\vdots\\
x_1\, x_2^{n-1}\\
x_2^n
\end{pmatrix} = \begin{pmatrix}
x_1^{n+1}\\
x_1^{n}\, x_2\\
\vdots\\
x_1^2\, x_2^{n-1}\\
x_1\,x_2^n
\end{pmatrix} = L_{n,1}\,\mathbb{X}_{n+1},
\end{equation*}
analogously, $x_2\,\mathbb{X}_n = L_{n,2}\,\mathbb{X}_{n+1}$. In addition, for $n \geqslant 1$,
\begin{equation*}
x_1^{-1}\,\mathbb{X}_{-n} = x_1^{-1} \begin{pmatrix}
x_2^{-n}\\
x_1^{-1}\, x_2^{-n+1}\\
\vdots\\
x_1^{-n+1}\, x_2^{-1}\\
x_1^{-n}
\end{pmatrix}= \begin{pmatrix}
x_1^{-1}x_2^{-n}\\
x_1^{-2}\, x_2^{-n+1}\\
\vdots\\
x_1^{-n}\, x_2^{-1}\\
x_1^{-n-1}
\end{pmatrix} = L_{n,2}\,\mathbb{X}_{-n-1},
\end{equation*}
and analogously, $x_2^{-1}\,\mathbb{X}_{-n} = L_{n,1}\,\mathbb{X}_{-n-1}.$

We can abbreviate above expressions in the form
\begin{equation*}
\mbox{for} \ \ n\geqslant1, \quad
x_i^{-1}\,\mathbb{X}_{-n} = L_{n,j}\,\mathbb{X}_{-(n+1)}, \quad i,j=1,2 \quad \mbox{and} \quad i \neq j.
\end{equation*}
also,
\begin{equation} \label{x12normal}
\mbox{for} \ \ n \geqslant 0, \quad
x_i \,\mathbb{X}_{n} = L_{n,i}\,\mathbb{X}_{n+1}, \quad i =1,2.  \qquad \qquad \qquad \qquad \
\end{equation}

\bigskip

A basis of $\Pi$ can be organized as sequences of polynomial vectors of increasing size whose entries are independent polynomials of the same total degree.

\begin{definition}[\cite{Ko82a, Ko82b}] A \emph{polynomial system (PS)} is a sequence of column polynomial vectors  $\{\mathbb{P}_n\}_{n\geqslant0}$ of increasing size $n+1$ in the form
$$
\mathbb{P}_n = \left(P_{n,0}(x_1,x_2), P_{n-1,1}(x_1,x_2), \ldots, P_{0,n}(x_1,x_2)\right)^T,
$$
such that every entry $P_{n-k,k}(x_1,x_2),$ $ k=0,1,\ldots,n,$ is a bivariate polynomial of total degree $n$, and $\{P_{n,0}(x_1,x_2), P_{n-1,1}(x_1,x_2), \ldots, P_{0,n}(x_1,x_2)\}$ is a linearly independent set.
\end{definition}

Observe that the set of entries of $\{\mathbb{P}_m\}_{m=0}^n$ is a basis of $\Pi_n$, and $\{\mathbb{X}_n\}_{n\geqslant0}$ defined as in \eqref{def_Xpos} is a PS. Therefore, every polynomials vector, $\mathbb{P}_n$, can be expressed as
\begin{equation}\label{expl_exp}
\mathbb{P}_n =
G_{n} \,\mathbb{X}_n + G_{n,n-1} \,\mathbb{X}_{n-1} +
 G_{n,n-2}\, \mathbb{X}_{n-2} + \cdots  + G_{n,1}\, \mathbb{X}_{1}  + G_{n,0} \, \mathbb{X}_{0},
\end{equation}
where $ G_{n,j}$ are matrices of size $(n+1) \times (j+1)$, and the square matrix $G_n=G_{n,n}$ is non-singular, and it is called the \emph{matrix leading coefficient of} $\mathbb{P}_n$.

Moreover, a PS is called \emph{monic} if every involved polynomial is monic as
\begin{equation*}\label{P_nk}
P_{n-k,k}(x_1,x_2) = x_1^{n-k}\,x_2^k + \sum_{i+j=0\atop i,j\geqslant0}^{n-1} \,c_{i,j}\,x_1^i\,x_2^j, \qquad c_{i,j}\in \mathbb{R}.
\end{equation*}
In that case, the matrix leading coefficient is the identity, $G_n = I_{n+1}$.

\bigskip

When computing a monic PS, every entry as in \eqref{P_nk} has $t_{n-1}= n + (n-1) + \cdots +2 +1$ coefficients $c_{i,j}\in\mathbb{R}$, and then,
the number of real coefficients in the monic polynomial vector
$\mathbb{P}_n$ is
\begin{equation} \label{number_coef}
t_{n-1} \times (n+1) = \frac{n(n+1)^2}{2}.
\end{equation}

\section{Bivariate varying orthogonality}
\label{sec_R_II}

Let $\Omega$ be a domain on the first quadrant of $\mathbb{R}^2$, i.e., $\Omega\subset \mathbb{R}^2_+ = \{(x_1, x_2)\in \mathbb{R}^2: x_1, x_2 > 0\}$, and let $W(x_1,x_2)$ be a weight function defined on $\Omega$ such that
\begin{equation*}
\mu_{k,m} = \iint_{\Omega}  x_1^{k} x_2^{m} \, W(x_1,x_2)\, dx_1\,dx_2 < \infty,
\end{equation*}
for $k,m =0, \pm 1, \pm 2, \ldots$. For every two polynomials $f, g \in \Pi$, we define the inner product
$$
\langle f,g\rangle = \iint_{\Omega}  f(x_1, x_2) \, g(x_1, x_2) \, W(x_1,x_2)\, dx_1\,dx_2.
$$

\bigskip

For a given weight function $ W(x_1,x_2)$, we are going to work with varying weight functions in the form
\begin{equation}\label{vw}
W_n(x_1,x_2) = x_1^{-n} x_2^{-n}  W(x_1,x_2),
\end{equation}
for $n\geqslant0$, that are positive on $\Omega\subset \mathbb{R}^2_+$. To this end, we start studying this kind of varying orthogonality.

Given a natural number $n\geqslant 0$, we define
\begin{align*}
\mu^{(n)}_{k,m} &=\iint_{\Omega} \, x_1^{k}\, x_2^{m} \, W_n(x_1,x_2)\, dx_1\,dx_2 \\
&= \iint_{\Omega} \, x_1^{-n+k}\, x_2^{-n+m} \, W(x_1,x_2)\, dx_1\,dx_2 = \mu_{-n+k,-n+m},
\end{align*}
for $k,m =0, \pm 1, \pm 2, \ldots$. Observe that we have defined the moments for the varying weight \eqref{vw}, and they  are well defined.

Associated with the varying weight $W_n(x_1,x_2)$, we define the $n$th varying inner product, in short \emph{VIP},
\begin{align}
\langle f,g\rangle_n &= \iint_{\Omega}  f(x_1, x_2) \, g(x_1, x_2) \, W_n(x_1,x_2)\, dx_1\,dx_2 \label{ip_n}\\
&= \iint_{\Omega}  f(x_1, x_2) \, g(x_1, x_2) \, x_1^{-n}\, x_2^{-n} \, W(x_1,x_2)\, dx_1\,dx_2.\nonumber
\end{align}
Obviously, $\langle \cdot,\cdot\rangle = \langle \cdot,\cdot\rangle_0$.

\subsection{Moment matrices}

For $n, r, s\geqslant 0$, we define the matrix $M_{r,s}^{(n)}$ of size $(r+1) \times (s+1)$, as
\begin{align}
M_{r,s}^{(n)} &= \langle \mathbb{X}_{r},\mathbb{X}_{s}^T\rangle_n =
\iint_{\Omega} \mathbb{X}_{r} \,\mathbb{X}_{s}^T  \, W_n(x_1,x_2)\, dx_1\,dx_2\nonumber\\
& =
\begin{pmatrix}
\mu^{(n)}_{r+s,0}  & \mu^{(n)}_{r+s-1,1} & \cdots & \mu^{(n)}_{r,s}\\[1ex]
\mu^{(n)}_{r+s-1,1}  & \mu^{(n)}_{r+s-2,2} & \cdots & \mu^{(n)}_{r-1,s+1}\\[1ex]
\vdots & \vdots &     &   \vdots \\[1ex]
\mu^{(n)}_{s,r}  & \mu^{(n)}_{s-1,r+1} & \cdots & \mu^{(n)}_{0,r+s}
\end{pmatrix}.\label{moment_s}
\end{align}
For $0\leqslant k\leqslant s$, we denote each column of $M_{r,s}^{(n)}$ by $\mathfrak{m}_{r,s,k}^{(n)}$, hence
\begin{equation*}
M_{r,s}^{(n)} = \begin{pmatrix}
\mathfrak{m}^{(n)}_{r,s,0} & \mathfrak{m}^{(n)}_{r,s,1} &  \cdots &  \mathfrak{m}_{r,s,s}^{(n)}
\end{pmatrix},
\end{equation*}
with
\begin{align}
\mathfrak{m}_{r,s,k}^{(n)} &=
\iint_{\Omega} \mathbb{X}_{r} \,x_1^{s-k}\,x_2^k  \,  W_n(x_1,x_2)\, dx_1\,dx_2
\nonumber \\
&= \begin{pmatrix}
\mu_{r+s-k,k}^{(n)} & \mu_{r+s-k-1,k+1}^{(n)} & \mu_{r+s-k-2,k+2}^{(n)} & \cdots & \mu_{s-k,k+r}^{(n)}\end{pmatrix}^T
\nonumber\\
&= \langle \mathbb{X}_{r}, x_1^{s-k}\,x_2^k\rangle_n.\label{columns_M}
\end{align}
Observe that, from the definition, the matrices $M_{r,s}^{(n)}$ satisfy
\begin{equation*}
M_{r,s}^{(n)} = (M_{s,r}^{(n)})^T.
\end{equation*}

Now we define the moment matrix $\mathcal{M}_n$ of size  $t_n \times t_n$, constructed by blocks $M_{r,s}^{(n)}$ for $0\leqslant r,s\leqslant n$, in the form
\begin{equation} \label{matrixMn}
\mathcal{M}_n =
\begin{pmatrix}
M_{0,0}^{(n)} & M_{0,1}^{(n)} & \cdots & M_{0,n}^{(n)} \\[1ex]
M_{1,0}^{(n)} & M_{1,1}^{(n)} & \cdots & M_{1,n}^{(n)} \\[1ex]
\vdots & \vdots &    & \vdots \\[1ex]
M_{n,0}^{(n)} & M_{n,1}^{(n)} & \cdots & M_{n,n}^{(n)}
\end{pmatrix},
\end{equation}
where $t_n$ was defined in \eqref{def.tn}.
\bigskip

Some examples  of $\mathcal{M}_n$,
$$
\mathcal{M}_0 = (M_{0,0}^{(0)}) = ( \mu_{0,0})
\quad \quad \mbox{of size} \ t_0 \times t_0 = 1 \times 1,
$$
$$
\mathcal{M}_1 = \begin{pmatrix}
M_{0,0}^{(1)} & M_{0,1}^{(1)}  \\[1ex]
M_{1,0}^{(1)} & M_{1,1}^{(1)}
\end{pmatrix}
=
\left(
\begin{array}{c|cc}
\mu^{(1)}_{0,0} & \mu^{(1)}_{1,0} & \mu^{(1)}_{0,1}     \\  \hline
\mu^{(1)}_{1,0}  & \mu^{(1)}_{2,0} & \mu^{(1)}_{1,1}     \\
\mu^{(1)}_{0,1}  & \mu^{(1)}_{1,1}  & \mu^{(1)}_{0,2}
\end{array}
\right)
\quad
$$
of size $t_1 \times t_1 = 3 \times 3$, and $\mathcal{M}_2$ of size  $t_2 \times t_2 = 6 \times 6$ is
\begin{align*}
\mathcal{M}_2 =
\begin{pmatrix}
M_{0,0}^{(2)} & M_{0,1}^{(2)} & M_{0,2}^{(2)} \\[1ex]
M_{1,0}^{(2)} & M_{1,1}^{(2)} & M_{1,2}^{(2)}  \\[1ex]
M_{2,0}^{(2)} & M_{2,1}^{(2)} & M_{2,2}^{(2)}
\end{pmatrix}
=\left(
\begin{array}{c|cc|ccc}
\mu^{(2)}_{0,0} & \mu^{(2)}_{1,0} & \mu^{(2)}_{0,1} & \mu^{(2)}_{2,0} & \mu^{(2)}_{1,1} & \mu^{(2)}_{0,2}  \\ \hline
\mu^{(2)}_{1,0} & \mu^{(2)}_{2,0} & \mu^{(2)}_{1,1 }& \mu^{(2)}_{3,0} & \mu^{(2)}_{2,1} & \mu^{(2)}_{1,2} \\
\mu^{(2)}_{0,1} & \mu^{(2)}_{1,1} & \mu^{(2)}_{0,2} & \mu^{(2)}_{2,1} & \mu^{(2)}_{1,2} & \mu^{(2)}_{0,3}  \\ \hline
\mu^{(2)}_{2,0} & \mu^{(2)}_{3,0} & \mu^{(2)}_{2,1} & \mu^{(2)}_{4,0} & \mu^{(2)}_{3,1} & \mu^{(2)}_{2,2}  \\
\mu^{(2)}_{1,1} & \mu^{(2)}_{2,1} & \mu^{(2)}_{1,2} & \mu^{(2)}_{3,1} & \mu^{(2)}_{2,2} & \mu^{(2)}_{1,3} \\
\mu^{(2)}_{0,2} & \mu^{(2)}_{1,2} & \mu^{(2)}_{0,3} & \mu^{(2)}_{2,2} & \mu^{(2)}_{1,3} & \mu^{(2)}_{0,4}
\end{array}
\right).
\end{align*}
From $\mathcal{M}_n$, we define the matrix  $\widehat{\mathcal{M}}_{n}$  of size $t_{n-1} \times t_{n-1}$ by deleting the last block row and column
$$
\widehat{\mathcal{M}}_{n}=
\begin{pmatrix}
M_{0,0}^{(n)} & M_{0,1}^{(n)} & \cdots & M_{0,n-1}^{(n)}\\[1ex]
M_{1,0}^{(n)} & M_{1,1}^{(n)} & \cdots & M_{1,n-1}^{(n)}\\[1ex]
\vdots & \vdots &   & \vdots   \\[1ex]
M_{n-1,0}^{(n)} & M_{n-1,1}^{(n)} & \cdots & M_{n-1,n-1}^{(n)}
\end{pmatrix}.
$$

We observe that $\mathcal{M}_n$ is the $n$th block moment matrix, and $\widehat{\mathcal{M}}_{n}$ in the $(n-1)$th block moment matrix associated with the weight function $W_n(x_1, x_2) = x_1^{-n}\,x_2^{-n}\,W(x_1, x_2)$. Consequently, they are symmetric and positive definite matrices (\cite[p. 64]{DX14}).

Hence
$$
\mathcal{M}_n =
\left( \begin{array}{cccc|c}
 & & & &  M_{0,n}^{(n)} \\[1ex]
 & & \widehat{\mathcal{M}}_{n} & &  M_{1,n}^{(n)} \\
 & & & &  \vdots  \\
 & & & &   M_{n-1,n}^{(n)}  \\ \hline
M_{n,0}^{(n)} & M_{n,1}^{(n)} &  \cdots & M_{n,n-1}^{(n)} & M_{n,n}^{(n)}
\end{array} \right).
$$

\subsection{Orthogonal polynomial systems with respect to varying weights}

Our objective is to construct a PS $\{\mathbb{P}_n\}_{n\geqslant0}$ satisfying orthogonality conditions by using varying weights. Let $n\geqslant 0$, then for $k=0,1,\ldots, n-1$,
\begin{align}
\langle \mathbb{X}_{k},\mathbb{P}_n^T\rangle_n &=\iint_{\Omega}  \mathbb{X}_{k} \,\mathbb{P}_n^T \,W_n(x_1,x_2)\, dx_1\,dx_2  \nonumber \\
&=
\iint_{\Omega}  \mathbb{X}_{k} \,\mathbb{P}_n^T \,x_1^{-n}\, x_2^{-n} \,W(x_1,x_2)\, dx_1\,dx_2 = \mathtt{0}_{(k+1) \times (n+1)}, \label{RII_ort_con}
\end{align}
where $\mathtt{0}_{(k+1) \times (n+1)}$ denotes a zero matrix of size $(k+1) \times (n+1)$, and
\begin{align}\label{k=n}
\langle \mathbb{X}_{n},\mathbb{P}_n^T\rangle_n =  \Lambda_n,
\end{align}
where  $\Lambda_n$ is a non-singular real matrix of size $(n+1) \times (n+1)$.

\begin{definition}
A PS satisfying the varying orthogonality conditions \eqref{RII_ort_con}-\eqref{k=n} will be called a polynomial system orthogonal with respect to the varying weight $W_n(x_1, x_2)$, in short \emph{VOPS}.
\end{definition}

\medskip

\subsection{Existence of VOPS}

From \eqref{expl_exp}, we get
$$
\mathbb{P}_n^T = \mathbb{X}_{0}^T G_{n,0}^T + \mathbb{X}_{1}^T G_{n,1}^T + \cdots  + \mathbb{X}_{n-1}^T G_{n,n-1}^T
+ \mathbb{X}_n ^T  G_n^T = \sum_{i=0}^n\,\mathbb{X}_{i}^T G_{n,i}^T,
$$
and, using the orthogonality conditions \eqref{RII_ort_con} and \eqref{k=n}, making $\langle \mathbb{X}_{k},\,\mathbb{P}_{n}^T\rangle_n$ for $k= 0,1, \cdots n-1, n$,  we can write
\begin{align*}
& \sum_{i=0}^n\langle \mathbb{X}_{0},\,\mathbb{X}_{i}^T\rangle_n\, G_{n,i}^T = \mathtt{0}_{1 \times (n+1)} \\
& \sum_{i=0}^n\langle \mathbb{X}_{1},\,\mathbb{X}_{i}^T\rangle_n\, G_{n,i}^T = \mathtt{0}_{2 \times (n+1)} \\
& \qquad \vdots \\
& \sum_{i=0}^n\langle \mathbb{X}_{n-1},\,\mathbb{X}_{i}^T\rangle_n\, G_{n,i}^T = \mathtt{0}_{n \times (n+1)} \\
& \sum_{i=0}^n\langle \mathbb{X}_{n},\,\mathbb{X}_{i}^T\rangle_n\, G_{n,i}^T =   \Lambda_n.
\end{align*}
Using \eqref{matrixMn} we get the matrix system
$$
\begin{pmatrix}
M_{0,0}^{(n)} & M_{0,1}^{(n)} & \cdots & M_{0,n}^{(n)} \\[1ex]
M_{1,0}^{(n)} & M_{1,1}^{(n)} & \cdots & M_{1,n}^{(n)} \\[1ex]
\vdots & \vdots &    & \vdots \\[1ex]
M_{n,0}^{(n)} & M_{n,1}^{(n)} & \cdots & M_{n,n}^{(n)}
\end{pmatrix}
\begin{pmatrix}
G_{n,0}^T \\[2ex]
G_{n,1}^T \\[2ex]
\vdots \\
G_{n}^T
\end{pmatrix}
=
\begin{pmatrix}
\mathtt{0}_{1 \times (n+1)}  \\[1ex]
\mathtt{0}_{2 \times (n+1)}  \\[1ex]
\vdots \\[1ex]
\Lambda_n
\end{pmatrix},
$$
whose coefficient matrix is $\mathcal{M}_n$, the moment matrix. Observe that $G_{n,k}^T$, for $k= 0, 1, \ldots n$ are unknown matrices of size $(k+1)\times (n+1)$. Then we have unique solution since the coefficient matrix is non-singular.

\medskip

Next result shows how to construct the monic VOPS.

\begin{theorem} \label{3.2}
For $n\geqslant0$ and $k=0,1,\ldots,n$, the entries of every vector $\mathbb{P}_n$ are given by
$$
P_{n-k,k}(x_1,x_2) = \frac{1}{ \det( \widehat{\mathcal{M}}_{n} ) }
\left| \begin{array}{cccc|c}
 & &  & &  \mathfrak{m}_{0,n,k}^{(n)}  \\
 & & \widehat{\mathcal{M}}_{n} & &  \mathfrak{m}_{1,n,k}^{(n)}  \\
 & &  &  & \vdots \\
 & & & &  \mathfrak{m}_{n-1,n,k}^{(n)}   \\ \hline
\mathbb{X}_{0}^T & \mathbb{X}_{1}^T & \cdots & \mathbb{X}_{n-1}^T    & x_1^{n-k} x_2^{k}
\end{array}\right|,
$$
where $
\mathfrak{m}_{r,n,k}^{(n)}
$
was defined in \eqref{columns_M}.
\end{theorem}

\begin{proof}

Obviously, $P_{n-k,k}(x_1,x_2)$ is a monic polynomial whose highest degree term is $x_1^{n-k} x_2^{k}$. Computing
$
\langle x_1^{m-j}\,x_2^j, P_{n-k,k}\rangle_n
$,
for $m<n$, and $0\leqslant j\leqslant m$, the last row of the the determinant coincides with another row above, therefore the determinant is zero, and the orthogonality holds.

\end{proof}

\section{R$_{II}$ type three term relations}
\label{sec_R_II_TTR}

We can now prove that the VOPS $\{\mathbb{P}_{n}\}_{n\geqslant0}$ that satisfy the orthogonality conditions \eqref{RII_ort_con}-\eqref{k=n}, also satisfy three term relations.

\begin{theorem}
Let $\{\mathbb{P}_{n}\}_{n\geqslant0}$ be a VOPS associated with the VIP \eqref{ip_n}. For $n\geqslant0$ and $i=1,2$, there exist matrices $A_{n,i}$, $B_{n,i}$, $C_{n,i}$, of respective sizes $(n+1) \times (n+2)$, $(n+1) \times (n+1)$ and $(n+1) \times n$, such that the following R$_{II}$ type three term relations hold
\begin{equation} \label{TTR_R_II}
x_i \,\mathbb{P}_n  = A_{n,i} \,\mathbb{P}_{n+1}
+ B_{n,i} \, \mathbb{P}_{n}
+ x_{1} \, x_{2} \, C_{n,i} \, \mathbb{P}_{n-1},  \quad i=1,2, \ n \geqslant 0,
\end{equation}
with $\mathbb{P}_{-1} = 0$, $\mathbb{P}_0 = (1),$
if and only if, for $n \geqslant 1$,
\begin{align}
&A_{n,i}\, G_{n+1}  =  G_{n}\, L_{n,i} - C_{n,i}\, G_{n-1}\, L_{n-1,1}\, L_{n,2}, \label{AGCL} \\[1ex]
&L_{n-1,1}\, L_{n,2}\, \Lambda_{n+1} \,A_{n,i}^{T}
= L_{n-1,i} \,\Lambda_{n}
- \Lambda_{n-1} \,C_{n,i}^{T},\label{RRC}
\end{align}
where $G_n$ is the leading coefficient matrix of $\mathbb{P}_n$, $L_{n,i}$ and $\Lambda_{n}$ are defined in \eqref{L_ni} and \eqref{RII_ort_con}, respectively.

Moreover, in this case, we have that
\begin{equation}
B_{n,i} \,\Upsilon_{n}^T  =  - C_{n,i}\, \Upsilon_{n-1}^T,
\quad n \geqslant 0,   \label{BLCL}
\end{equation}
for $i=1,2$, where
$$
\Upsilon_n = \langle \mathbb{X}_{-1},\mathbb{P}_n^T\rangle_n, \quad n \geqslant 0.
$$
\end{theorem}
\begin{proof} For $n\geqslant0$ and $i=1,2$, let $A_{n,i}$ and $C_{n,i}$ be $(n+1) \times (n+2)$ and $(n+1) \times n$ matrices, respectively, satisfying \eqref{AGCL}. We define the $(n+1) \times 1$ polynomial vector
$$
\mathbb{T}_i= A_{n,i} \,\mathbb{P}_{n+1}   - x_i \,\mathbb{P}_n +  x_{1} \, x_{2} \, C_{n,i} \, \mathbb{P}_{n-1},
$$
of degree $n+1$. From the explicit expression \eqref{expl_exp}, we can write
\begin{equation*}
 \mathbb{T}_i = A_{n,i} \sum_{k=0}^{n+1}  G_{n+1,k} \, \mathbb{X}_{k}
  - x_i  \sum_{k=0}^{n}  G_{n,k} \, \mathbb{X}_{k}
  + x_1 \, x_2   C_{n,i} \sum_{k=0}^{n-1} G_{n-1,k} \, \mathbb{X}_{k}.
\end{equation*}
Using \eqref{x12normal}, we get
\begin{equation*}
\begin{aligned}
 \mathbb{T}_i =& A_{n,i} \sum_{k=0}^{n+1}  G_{n+1,k} \mathbb{X}_{k}
  - \sum_{k=0}^{n}  G_{n,k} L_{k,i} \mathbb{X}_{k+1}
  +  C_{n,i} \sum_{k=0}^{n-1}  G_{n-1,k} L_{k,1} L_{k+1,2} \mathbb{X}_{k+2} \\[1ex]
 = &\left[ A_{n,i} G_{n+1} - G_{n}L_{n,i}  +  C_{n,i}  G_{n-1} L_{n-1,1} L_{n,2} \right] \mathbb{X}_{n+1} \\
 + &
 A_{n,i} \sum_{k=0}^{n}  G_{n+1,k} \mathbb{X}_{k}
  - \sum_{k=0}^{n-1}  G_{n,k} L_{k,i} \mathbb{X}_{k+1}
  +  C_{n,i} \sum_{k=0}^{n-2}  G_{n-1,k} L_{k,1} L_{k+1,2} \mathbb{X}_{k+2}.
\end{aligned}
\end{equation*}
Then, from \eqref{AGCL}, $ \mathbb{T}_i$ has degree $n$. Therefore, since $\{\mathbb{P}_k\}_{k=0}^n$ is a basis of $\Pi_n$, there exist matrices $F_{k,i}$ of size  $(n+1)  \times (k+1)$, $i=1,2$, such that
$$
\mathbb{T}_i = \sum_{k=0}^{n} F_{k,i} \mathbb{P}_k.
$$
Now we denote $F_{n,i} = - B_{n,i}$, hence
$$
\mathbb{T}_i = A_{n,i} \,\mathbb{P}_{n+1}  - x_i \, \mathbb{P}_n
+ x_1 \, x_2 \, C_{n,i} \, \mathbb{P}_{n-1}
= - B_{n,i} \mathbb{P}_n + \sum_{k=0}^{n-1} F_{k,i} \mathbb{P}_k,
$$
and therefore we can write
\begin{equation}\label{eq_Q}
A_{n,i} \,\mathbb{P}_{n+1}  - x_i \, \mathbb{P}_n
+ B_{n,i} \mathbb{P}_n
+ x_1 \, x_2 \, C_{n,i} \, \mathbb{P}_{n-1}
= \mathbb{Q}_{n-1,i}
\end{equation}
where
$$
\mathbb{Q}_{n-1,i} = \sum_{k=0}^{n-1} F_{k,i} \mathbb{P}_k
$$
is a polynomial vector of degree $n-1$ and of size $(n+1) \times 1$. Our objective is to prove that $ \mathbb{Q}_{n-1,i}$ vanishes.

Using the transpose of equation \eqref{eq_Q}, multiplying it by $\mathbb{X}_{k}$, for $k=0,1,\ldots,n-1$, and applying the $n$th VIP, we get
\begin{align*}
\langle \mathbb{X}_{k}, \mathbb{P}_{n+1}^{T} \rangle_n \, A_{n,i}^{T}
- \langle \mathbb{X}_{k}, x_i \, \mathbb{P}_n^{T}\rangle_n  & + \langle \mathbb{X}_{k}, \mathbb{P}_n^{T}\rangle_n\, B_{n,i}^{T}\\
& + \langle \mathbb{X}_{k}, x_1 \, x_2 \, \mathbb{P}_{n-1}^{T}\rangle_n \,  C_{n,i}^{T} = \langle \mathbb{X}_{k}, \mathbb{Q}_{n-1,i}^{T} \rangle_n.
\end{align*}
Taking into account that $x_1^{-1}\,x_2^{-1}\,W_n = W_{n+1}$, and $x_1\,x_2\,W_n = W_{n-1}$, we have
\begin{align*}
\langle x_1\,x_2\,\mathbb{X}_{k}, \mathbb{P}_{n+1}^{T} \rangle_{n+1} \, A_{n,i}^{T}
& - \langle x_i \,\mathbb{X}_{k}, \mathbb{P}_n^{T}\rangle_n   + \langle \mathbb{X}_{k}, \mathbb{P}_n^{T}\rangle_n\, B_{n,i}^{T}\\
& + \langle \mathbb{X}_{k}, \mathbb{P}_{n-1}^{T}\rangle_{n-1} \,  C_{n,i}^{T} = \langle \mathbb{X}_{k}, \mathbb{Q}_{n-1,i}^{T} \rangle_n,
\end{align*}
and from \eqref{x12normal} we can write
\begin{equation} \label{intint_2}
\begin{aligned}
 L_{k,1} L_{k+1,2}\langle & \mathbb{X}_{k+2}, \mathbb{P}_{n+1}^{T} \rangle_{n+1}  A_{n,i}^{T}  - L_{k,i} \langle \mathbb{X}_{k+1}, \mathbb{P}_n^{T}\rangle_n  \\[1ex]
& + \langle \mathbb{X}_{k}, \mathbb{P}_n^{T}\rangle_n\, B_{n,i}^{T}
+ \langle \mathbb{X}_{k}, \mathbb{P}_{n-1}^{T}\rangle_{n-1} \,  C_{n,i}^{T}
 = \langle \mathbb{X}_{k}, \mathbb{Q}_{n-1,i}^{T} \rangle_n.
\end{aligned}
\end{equation}
Using \eqref{RII_ort_con}, we observe that all terms in the left hand side vanish for  $k=  0,1,\ldots,n-2$, and then
$$
 \langle \mathbb{X}_{k}, \mathbb{Q}_{n-1,i}^{T} \rangle_n = \mathtt{0}_{(k+1) \times (n+1)}.
$$

\bigskip

Substituting $k=n-1$ in \eqref{intint_2} and using \eqref{RII_ort_con}  and \eqref{k=n},  we obtain
$$
  L_{n-1,1}\, L_{n,2} \,\Lambda_{n+1}\, A_{n,i}^{T}
- L_{n-1,i}\, \Lambda_{n}
+ \Lambda_{n-1}\, C_{n,i}^{T} =
 \langle \mathbb{X}_{n-1}, \mathbb{Q}_{n-1,i}^{T} \rangle_n.
$$
Therefore, \eqref{RRC} for $i=1,2$, provides the square homogeneous linear system of size $(n+1) t_{n-1}$
\begin{equation}  \label{eq_sist}
\langle \mathbb{X}_{k}, \mathbb{Q}_{n-1,i}^{T} \rangle_n = \mathtt{0}_{(k+1) \times (n+1)},  \quad k=0, 1,2,\ldots,n-1.
\end{equation}

Now we look at the matrix of the  homogeneous system \eqref{eq_sist}. Setting
$$
\mathbb{Q}_{n-1,i} = \sum_{r=0}^{n-1} E_{r,i} \mathbb{X}_r \quad \mbox{and} \quad
\mathbb{Q}_{n-1,i}^T = \sum_{r=0}^{n-1}\mathbb{X}_r^T  E_{r,i}^T,
$$
where $E_{r,i}$ are constant $(n+1)\times (r+1)$ matrices, we write \eqref{eq_sist} as
$$
\sum_{r=0}^{n-1}\langle \mathbb{X}_{k}, \mathbb{X}_r^T \rangle_n E_{r,i}^T =  \mathtt{0}_{(k+1) \times (n+1)}, \quad k=0, 1, \ldots,n-1.
$$
Using the definition of $M_{r,s}^{(n)}$ given in \eqref{moment_s}, we can write the above linear system of equations as
$$
\begin{pmatrix}
M_{0,0}^{(n)} & M_{0,1}^{(n)} & \cdots & M_{0,n-1}^{(n)}\\[1ex]
M_{1,0}^{(n)} & M_{1,1}^{(n)} & \cdots & M_{1,n-1}^{(n)}\\[1ex]
\vdots & \vdots &   & \vdots   \\[1ex]
M_{n-1,0}^{(n)} & M_{n-1,1}^{(n)} & \cdots & M_{n-1,n-1}^{(n)}
\end{pmatrix}
\begin{pmatrix}
E_{0,i}^T    \\[1ex]
E_{1,i}^T    \\[1ex]
\vdots \\[1ex]
E_{n-1,i}^T
\end{pmatrix}
=
\begin{pmatrix}
\mathtt{0}_{n \times (n+1)}  \\[1ex]
\mathtt{0}_{(n-1) \times (n+1)}      \\[1ex]
\vdots \\[1ex]
\mathtt{0}_{1 \times (n+1)}
\end{pmatrix}.
$$
The coefficient matrix in the above homogeneous system is $\widehat{\mathcal{M}}_{n}$ defined in \eqref{matrixMn}, and, since $\det \widehat{\mathcal{M}}_{n}>0$, the unique solution of the system is zero and $\mathbb{Q}_{n-1,i}\equiv \mathtt{0}$. Therefore, the R$_{II}$ type three term relations \eqref{TTR_R_II} hold.

The reciprocal can be showed in an analogous way.

\bigskip

Finally, we can take $k=-1$ in \eqref{intint_2}. Since all integrals are well defined for $n\geqslant0$, we deduce
\begin{equation*}
\langle \mathbb{X}_{-1}, \mathbb{P}_n^{T}\rangle_n \,B_{n,i}^{T}
+ \langle \mathbb{X}_{-1}, \mathbb{P}_{n-1}^{T} \rangle_{n-1} \,  C_{n,i}^{T} = \mathtt{0}_{2 \times (n+1)},
\end{equation*}
and \eqref{BLCL} holds.
\end{proof}

\subsection{R$_{II}$ type three term recurrence relations}

Observe that the R$_{II}$ type three term relations given by \eqref{TTR_R_II} do not play a rule of recurrence relation,  since the matrices $A_{n,i}$, $i=1,2$ are non squared and they can not be invertible.
Next result brings R$_{II}$ type three term recurrence relations for bivariate polynomials. Given $M_1$, $M_2$, two matrices of the same size, we define the \emph{joint matrix} as
$$
M=\begin{pmatrix}
M_{1} \\
M_{2}
\end{pmatrix}.
$$

\begin{theorem}  \label{Theorem_TTRR}
Let $\{\mathbb{P}_n\}_{n\geqslant0}$ be a VOPS satisfying the three term relations \eqref{TTR_R_II}. If the matrices $A_{n,i}$, $i=1,2$, and the joint matrix $A_{n}$, have full rank, then
\begin{equation} \label{TTRR_R_II}
\mathbb{P}_{n+1} = x_1 D_{n,1}^T \,\mathbb{P}_{n}
+ x_2 D_{n,2}^T \, \mathbb{P}_{n}
+ E_n \, \mathbb{P}_{n}
+  x_1 \, x_2 \, F_n \, \mathbb{P}_{n-1},  \quad \ n \geqslant 0,
\end{equation}
with $\mathbb{P}_{-1} = 0$, $\mathbb{P}_{0} = (1)$, where
$
D_n^T = \left(  D_{n,1}^T \  D_{n,2}^T \right),
$
such that
$
D_{n,1}^T \, A_{n,1} + D_{n,2}^T \, A_{n,2} = I_{n+2},
$ (it is the pseudo inverse of $A_n$) and
$$
E_n = -\sum_{i=1}^{2}  D_{n,i}^T \, B_{n,i}, \quad \mbox{and} \quad
F_n = -\sum_{i=1}^{2}  D_{n,i}^T \, C_{n,i}.
$$
\end{theorem}

\begin{proof}
Since the joint matrix $A_n$ of size $2(n+1)\times (n+2)$ has full rank, then there exists a (not unique) full rank matrix $D_n^T = \left(  D_{n,1}^T \  D_{n,2}^T \right)$ of size $(n+2)\times 2(n+1)$ such that
$$
D_n^T \,A_n = D_{n,1}^T\,A_{n,1} + D_{n,2}^T\,A_{n,2} = I_{n+2}.
$$
Therefore, multiplying \eqref{TTR_R_II} for $i=1$ by $D_{n,1}^T$, and for $i=2$ by $D_{n,2}^T$  by the left hand side, and sum, we get
\begin{equation*}
\sum_{i=1}^2 x_i D_{n,i}^T \mathbb{P}_n =
 \mathbb{P}_{n+1} +
\sum_{i=1}^2 \, D_{n,i}^T \, B_{n,i} \,\mathbb{P}_{n}
+ x_1 \,x_2 \,\sum_{i=1}^2  D_{n,i}^T\,  C_{n,i}\,\mathbb{P}_{n-1},
\end{equation*}
and   \eqref{TTRR_R_II} holds.
\end{proof}

\section{A method for generating VOPS}
\label{sec_Koorwinder}
\setcounter{equation}{0}

Inspired by the well known Koornwinder's method (see Dunkl and Xu \cite{DX14} and Koornwinder \cite{Ko75}) used to obtain bivariate orthogonal polynomials from univariate orthogonal polynomials, we develop a similar construction to obtain VOPS.

\medskip

Let us consider two weight functions in one variable, $w_1(x_1)$ defined in $(a,b)\subset \mathbb{R}^+$ and $w_2(x_2)$ in $(c,d)\subset \mathbb{R}^+$, and let $\rho$ be a positive polynomial of degree one in  $(a,b)$. We can construct a weight function in two variables as
$$
W(x_1,x_2)=w_1(x_1) w_2\left(\frac{x_2}{\rho(x_1)}\right),
$$
defined on
$$
\Omega=\left\{(x_1,x_2) \, :\, a<x_1<b, \, c\rho(x_1)<x_2<d\rho(x_1) \right\} \subset \mathbb{R}^2_+.
$$

For a fixed $n$ and a fixed $k$, we consider the sequences of monic orthogonal polynomials, $\{p_{m}^{(n,k)}\}_{m\geqslant 0}$ and $\{q_m^{(n)}\}_{m\geqslant 0}$, with respect to the weight functions
$$
\frac{\rho(x_1)^{2k+1} w_1(x_1)}{\rho(x_1)^{n} x_1^n} \quad \text{and} \quad \frac{w_2(x_2)}{x_2^n}\,,
$$
respectively.
We can write VOPS associated with $W(x_1,x_2)$
$$
\mathbb{P}_n=(P_{n,0}(x_1,x_2),P_{n-1,1}(x_1,x_2),\dots,P_{0,n}(x_1,x_2))^T
$$
as
$$
P_{n-k,k}(x_1,x_2)=p_{n-k}^{(n,k)}(x_1)\, \rho(x_1)^k q_k^{(n)}\left(\frac{x_2}{\rho(x_1)}\right).
$$
To prove the orthogonality \eqref{RII_ort_con}, let $m=0,1,\dots, n-1$ and $j \leqslant m$, then
$$
\begin{aligned}
I & = \iint_{\Omega}
x_1^{m-j}  x_2^j
P_{n-k,k}(x_1,x_2)
W_n(x_1,x_2) d x_1 d x_2
\\
&= \iint_{\Omega}
x_1^{m-j} x_2^j \
p^{(n,k)}_{n-k}(x_1) \rho(x_1)^k
q^{(n)}_k \left(\frac{x_2}{\rho(x_1)}\right)
\frac{w_1(x_1)}{x_1^n} \frac{w_2\left(\frac{x_2}{\rho(x_1)}\right)}{x_2^n} d x_1 d x_2,
\end{aligned}
$$
Setting $\displaystyle t= x_2/\rho(x_1)$,
since $c\rho(x_1) < x_2 < d \rho(x_1)$, then $c<t<d$ and the integral can be expressed as
$$
I=\int_{c}^{d}
\int_{a}^{b}
x_1^{m-j} t^j \rho(x_1)^{j+k+1}
p^{(n,k)}_{n-k}(x_1)
q^{(n)}_k (t)\frac{w_1(x_1)}{\rho(x_1)^n x_1^n}
\frac{w_2(t)}{t^n} d x_1 d t,
$$
and we can split it into two integrals
$$
I=\int_{a}^{b} x_1^{m-j} \, p^{(n,k)}_{n-k}(x_1)\frac{\rho(x_1)^{j+k+1}w_1(x_1)}{\rho(x_1)^n x_1^n} d x_1
\int_{c}^{d}  t^j  \, q^{(n)}_k (t) \frac{w_2(t)}{t^n}   d t.
$$
When $j<k$,  since the second integral in $I$ vanishes, then $I=0$. In the case $j=k$, we get
$$
\int_{c}^{d} t^k  \, q^{(n)}_k (t) \frac{w_2(t)}{t^n}   d t = C_k \neq 0,
$$
and we can write
$$
I=C_k \int_{a}^{b} x_1^{m-k} \, p^{(n,k)}_{n-k}(x_1) \frac{ \rho(x_1)^{2k+1}w_1(x_1)}{\rho(x_1)^n x_1^n} d x_1.
$$
Using the orthogonality of $p^{(n,k)}_{n-k}(x_1)$, since $m-k<n-k$, we conclude that also $I=0$.

When $j>k$, we can write
$$
I=\int_{a}^{b} x_1^{m-j} p^{(n,k)}_{n-k}(x_1) \rho(x_1)^{j-k} \frac{ \rho(x_1)^{2k+1}w_1(x_1)}{\rho(x_1)^n x_1^n} d x_1
\int_{c}^{d}  t^j  \, q^{(n)}_k (t) \frac{w_2(t)}{t^n}   d t,
$$
we observe that the first integral vanishes, since  $x_1^{m-j} \rho(x_1)^{j-k}$ is a polynomial of degree $m-k<n-k$.

\section{Examples}

\subsection{VOPS in a rectangle}

We consider the weight function in one variable
$$
w(t) = \frac{1}{\sqrt{t-a}\sqrt{b-t}}
$$
defined in $(a,b)$, where $0<a<b<\infty$.
Setting
$\beta = \sqrt{ab}$, the moments associated with $w$ satisfy
$ \mu_k = \beta^{2k+1} \mu_{-(k+1)}$ (see \cite{Ra94}).

\medskip

Then, considering two weight functions $ w_1(x_1)$ and $w_2(x_2)$ defined on positive intervals $(a,b)$ and $(c,d)$, respectively, we can construct the weight function in two variables as
$$
W(x_1,x_2) = w_1(x_1)w_2(x_2) =
\frac{1}{\sqrt{x_1-a}\sqrt{b-x_1}} \
\frac{1}{\sqrt{x_2-c}\sqrt{d-x_2}}
$$
defined on the rectangular region
$$\Omega=\left\{(x_1,x_2) \, :\, 0<a<x_1<b<\infty, \,  0< c<x_2<d < \infty \right\}.
$$
The Figure~\ref{fig:rectangle} shows the region  $\Omega.$
\begin{center}
	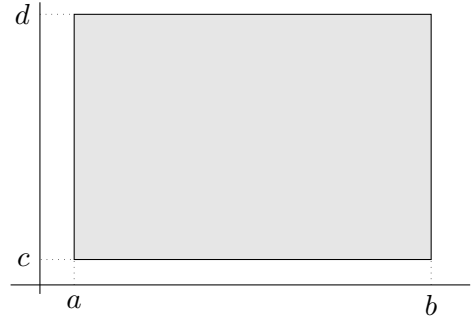
\begin{figure}[h]
\begin{tikzpicture}[xscale=1.3,yscale=0.4,scale=1]
	\tikzmath{\xa=1/4+0.1; \xb=4; \yc=4/9+0.4; \yd=9;}
	\draw (-0.3,0) -- (\xb+0.4,0);
	\draw (0,-0.3) -- (0,\yd+0.4);
	\coordinate [label=below:{$a$}]  (a) at (\xa,0);
	\coordinate [label=below:{$b$}]  (b) at (\xb,0);
	\coordinate [label=left:{$c$}]  (c) at (0,\yc);
	\coordinate [label=left:{$d$}]  (d) at (0,\yd);
	\coordinate (ac) at (\xa,\yc);
	\coordinate (ad) at (\xa,\yd);
	\coordinate (bc) at (\xb,\yc);
	\coordinate (bd) at (\xb,\yd);
	\draw[very thin, dotted] (a) -- (ac) -- (c);
	\draw[very thin, dotted] (d) -- (ad);
	\draw[very thin, dotted] (b) -- (bc);
	\filldraw[fill=black!10] (ac) -- (ad) -- (bd) -- (bc) -- cycle;
\end{tikzpicture}
		\caption{Rectangular region.}\label{fig:rectangle}
\end{figure}
\end{center}

\medskip

Considering
$\beta_1 = \sqrt{ab}$ and
$\beta_2 = \sqrt{cd},$
then the moments of $w_1(x_1)$ and $w_2(x_2)$ satisfy
$
\mu^{(1)}_r = \beta_1^{2r+1} \mu^{(1)}_{-(r+1)}
$ and
$
\mu^{(2)}_s = \beta_2^{2s+1} \mu^{(2)}_{-(s+1)},
$
respectively.
Hence, we observe that the moments associated with $W(x_1,x_2)$ satisfy the identity
$
\mu_{r,s} = \beta_1^{2r+1}  \beta_2^{2s+1} \mu_{-(r+1),-(s+1)}.
$

\medskip

For a fixed $n$, we consider the sequences of monic orthogonal polynomials, $\{ p^{(n)}_m \}_{m\geqslant 0}$ and
$\{ q^{(n)}_m \}_{m\geqslant 0}$, with respect to the weight functions
$$ \frac{w_1(x_1)}{x_1^n} \quad \mbox{and} \quad \frac{w_2(x_2)}{x_2^n},$$
respectively.  Hence we can write  the monic VOPS associated with  $W(x_1,x_2)$,
$$
\mathbb{P}_n = \left(P_{n,0}(x_1,x_2), P_{n-1,1}(x_1,x_2), \ldots, P_{0,n}(x_1,x_2)\right)^T,
$$
as
$$
P_{n-k,k}(x_1,x_2) = p^{(n)}_{n-k}(x_1) q^{(n)}_k (x_2),
$$
since
$$
\int_{c}^{d}\int_{a}^{b}
x_1^j \, x_2^h \,
p^{(n)}_{n-k}(x_1) \, q^{(n)}_k (x_2) \,
\frac{w(x_1)}{x_1^n}
\frac{w(x_2)}{x_2^n} d x_1 d x_2 = 0,
$$
for $0 \leqslant j < n-k \  \mbox{and} \
0 \leqslant h < k.$

In order to give some numerical computations, let us choose specific values of the parameter, as $a=1/4$, $b=4$, $c=4/9$, and $d=9$. Then the weight function becomes
\begin{equation}
\label{wf_ret}
W(x_1,x_2) =  \frac{1}{\sqrt{x_1-\frac{1}{4}}\sqrt{4-x_1}} \frac{1}{\sqrt{x_2-\frac{4}{9}}\sqrt{9-x_2}},
\end{equation}
defined on $\Omega = \{(x_1, x_2) \, : \, 1/4< x_1 <4, \ 4/9< x_2<9 \}.$

Table \ref{tOP1var} shows the sequence of monic orthogonal polynomials in one variable  $\{ p^{(n)}_m \}_{m\geqslant 0}$ and
$\{ q^{(n)}_m \}_{m\geqslant 0}$, with respect to the weight functions
$ 1/(\sqrt{x_1-1/4} \, \sqrt{4-x_1} \, x_1^n)$ and
$ 1/(\sqrt{x_2-4/9} \, \sqrt{9-x_2}\, x_2^n),$
respectively.

The entries of the first VOPS with respect the weight function \eqref{wf_ret} are given in Table \ref{tVOPS_ex1}.

\begin{table}[ht!]
\begin{tabular}{ccccc}
$k$ & $0$ & $1$ & $2$ & $3$ \\
\hline \hline
$p^{(1)}_k(x_1)$ & $1$ & $x_1-1$\\[1ex]
$q^{(1)}_k(x_2)$ & $1$ & $x_2-2$\\[1ex]
\hline
$p^{(2)}_k(x_1)$ & $1$ & $x_1-\frac{8}{17}$  & $x_1^2 - \frac{25}{8}x_1 +1$\\[2ex]
$q^{(2)}_k(x_2)$ & $1$ & $x_2-\frac{72}{85}$ & $x_2^2 - \frac{121}{18}x_2+4$\\[1.5ex]
\hline
$p^{(3)}_k(x_1)$ & $1$ & $x_1 - \frac{272}{803}$  & $x_1^2 - \frac{59}{25}x_1 + \frac{16}{25}$ & $x_1^3 - \frac{75}{16} x_1^2 + \frac{75}{16}x_1 - 1$\\[2ex]
$q^{(3)}_k(x_2)$ & $ \ 1 \ $ & $ \ x_2 - \frac{4080}{6793} \ $ & $ \ x_2^2 - \frac{582}{121}x_2 + \frac{288}{121} \ $ & $ \ x_2^3 - \frac{121}{12} x_2^2 + \frac{121}{6}x_2 - 8 \ $\\[1.5ex]
\hline\\
\end{tabular}
\caption{Table of the monic orthogonal polynomials $p^{(n)}_k(x_1)$ and  $q^{(n)}_k(x_2)$, for $1 \leqslant n \leqslant 3$ and degree $0 \leqslant k \leqslant n$.}
\label{tOP1var}
\end{table}

\begin{table}[ht!]
\begin{tabular}{ll}
\hline \hline
$P_{0,0}$ & $1$\\[1ex]
\hline
$P_{1,0}$ & $x_1-1$\\[1ex]
$P_{0,1}$ & $x_2-2$\\[.5ex]
\hline
$P_{2,0}$ & $x_1^2 - \frac{25}{8}x_1 + 1$\\[2ex]
$P_{1,1}$ & $x_1x_2 - \frac{72}{85}x_1  - \frac{8}{17}x_2 + \frac{576}{1445}$\\[2ex]
$P_{0,2}$ & $x_2^2  - \frac{121}{18}x_2 + 4$\\[1.5ex]
\hline
$P_{3,0}$ & $x_1^3 - \frac{75}{16} x_1^2 + \frac{75}{16}x_1 - 1$\\[2ex]
$P_{2,1}$ & $x_1^2x_2  - \frac{4080}{6793} x_1^2
- \frac{59}{25}x_1x_2
+ \frac{48144}{33965}x_1
+ \frac{16}{25}x_2
- \frac{13056}{33965}$\\[2ex]
$P_{1,2}$ & $x_1x_2^2  - \frac{582}{121} x_1x_2
- \frac{272}{803} x_2^2
+ \frac{288}{121}x_1
+ \frac{158304}{97163}x_2
- \frac{78336}{97163}$\\[2ex]
$P_{0,3}$ & $x_2^3 - \frac{121}{12} x_2^2 + \frac{121}{6}x_2 - 8$\\[1.5ex]
\hline \\
\end{tabular}
\caption{Entries $P_{n-k,k}$ of the first VOPS with respect the weight function \eqref{wf_ret} for degree $0 \leqslant n \leqslant 3$  and for  $0 \leqslant k \leqslant n$.}
\label{tVOPS_ex1}
\end{table}

In this case the matrices in \eqref{k=n},
$ \Lambda_n = \langle \mathbb{X}_{n},\mathbb{P}_n^T\rangle_n$ are diagonal, for example for $n=3$ we get
$$
\Lambda_3 = \langle \mathbb{X}_{3},\mathbb{P}_3^T\rangle_3
=
\begin{pmatrix}
\frac{183411}{524288} & 0 & 0 & 0 \\
0 & \frac{160083}{2717200} & 0 & 0 \\
0 & 0 & \frac{60025}{194326} & 0  \\
0 & 0 & 0 & \frac{94472147}{2985984}
\end{pmatrix}.
$$
\subsection{VOPS in a triangle}

Using the construction described in Section \ref{sec_Koorwinder}, if we choose
$$
w_1(x_1) = (x_1-a)^\alpha  (b-x_1)^{\beta+\gamma} \quad \mbox{and} \quad w_2(x_2) = (x_2-c)^\beta  (d-x_2)^{\gamma},
$$
defined in $(a,b)$ and $(c,d)$, respectively, and $\rho(x) = b-x_1,$
then the weight function in two variables takes the form
\begin{eqnarray*}
W(x_1,x_2)
&=& (x_1-a)^\alpha  (b-x_1)^{\beta+\gamma}
\left(\frac{x_2}{b-x_1}-c\right)^\beta
\left(d-\frac{x_2}{b-x_1}\right)^{\gamma} \\
&=& (x_1-a)^\alpha
\left(x_2-c(b-x_1)\right)^\beta
\left(d(b-x_1)-x_2\right)^{\gamma},
\end{eqnarray*}
defined in the triangular region
$$
\Omega=\left\{(x_1,x_2) \, :\, a<x_1<b, \, c(b-x_1)<x_2<d(b-x_1) \right\},
$$
as in Figure~\ref{fig:triangle}.

\begin{center}
	\begin{figure}[h]
\begin{tikzpicture}[xscale=1,yscale=1,scale=0.7]
	\tikzmath{\xa=1;\xb=8;\yc=2.4;\yd=6;}
	\draw (-0.3,0) -- (\xb+0.5,0);
	\draw (0,-0.3) -- (0,\yd+0.5);
	\coordinate [label=below:{$a$}]  (a) at (\xa,0);
	\coordinate [label=below:{$b$}]  (b) at (\xb,0);
	\coordinate [label=left:{$c(b-a)$}]  (c) at (0,\yc);
	\coordinate [label=left:{$d(b-a)$}]  (d) at (0,\yd);
	\coordinate (ac) at (\xa,\yc);
	\coordinate (ad) at (\xa,\yd);
	\draw[very thin, dotted] (a) -- (ac) -- (c);
	\draw[very thin, dotted] (d) -- (ad);
	\filldraw[fill=black!10] (b) -- (ac) -- (ad) -- cycle;
\end{tikzpicture}
		\caption{Triangular region.}\label{fig:triangle}
\end{figure}
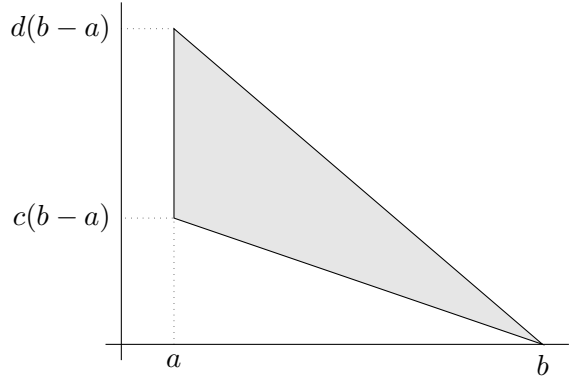
\end{center}

\medskip

Choosing $\alpha=1, \beta=2, \gamma=1, a=1$, $b=2$, $c=3$ and $d=5$, the weight functions become
$$
w_1(x_1) = (x_1-1)(2-x_1)^3,  \quad x_1 \in (1,2),
$$
$$
w_2(x_2) = (x_2-3)^2 (5-x_2),  \quad x_2 \in (3,5),
$$
and
\begin{equation} \label{ex_koor}
W(x_1,x_2)= \tau \, (x_1-1)\, (x_2+3x_1-6)^2 \, (10-5x_1-x_2),
\end{equation}
defined in $\Omega = \{(x_1, x_2)   \, : \, 1 < x_1 < 2,  \, 6-3x_1 < x_2 < 10 -5x_1 \}$. Here $\tau = 45/2$ is chosen so that
$\mu_{0,0}=1$.

Since
$$
P_{n-k,k}(x_1,x_2)=p_{n-k,k}^{(n)}(x_1)\, \rho(x_1)^k q_k^{(n)}\left(\frac{x_2}{\rho(x_1)}\right),
$$
for fixed $n$ and $k$, we need to construct the orthogonal polynomials, $p_{n-k,k}^{(n)} $ and $q_k^{(n)}$,  with respect to the weight functions
$$
\frac{(x_1-1)(2-x_1)^{3+2k+1-n}}{ \, x_1^n} \quad \text{and} \quad \frac{(x_2-3)^2 (5-x_2)}{x_2^n}\,,
$$
respectively.

Table \ref{tVOPS_Koor} shows the entries of the VOPS with respect the weight function given in \eqref{ex_koor}.

\begin{table}[ht!]
\begin{tabular}{ll}
\hline \hline
$P_{0,0}$ & {\small $1$}\\[1ex]
\hline
$P_{1,0}$ & {\small $x_1 - 1.31041371$}\\[1ex]
$P_{0,1}$ & {\small $x_2
+ 4.16034291 x_1
- 8.32068582$}\\[.5ex]
\hline
$P_{2,0}$ & {\small $x_1^2
- 2.78649820 x_1
+ 1.90208670$}\\[1ex]
$P_{1,1}$ & {\small $x_1x_2
+ 4.11932686 x_1^2
- 13.3840858 x_1
- 1.24909536 x_2
+ 10.2908642$}\\[1ex]
$P_{0,2}$ & {\small $x_2^2
+ 8.18284423 x_1 x_2
+ 16.5699978 x_1^2
- 66.2799910 x_1
- 16.3656885 x_2$}\\
&  {\small
$+ 66.2799910$}\\[.5ex]
\hline
$P_{3,0}$ & {\small $x_1^3
- 4.37745730 x_1^2
+ 6.28003391 x_1
- 2.95184419$} \\[1ex]
$P_{2,1}$ & {\small  $x_1^2 x_2
+ 4.07720565 x_1^3
- 18.9879650 x_1^2
- 2.65710257 x_1 x_2
+ 28.7424743 x_1$} \\[.5ex]
&  {\small
$+ 1.73534712 x_2
- 14.1507342$}\\[1ex]
$P_{1,2}$ & {\small  $x_1 x_2^2
+ 8.13057120 x_1^2 x_2
+ 16.3546681 x_1^3
- 85.1393062 x_1^2
- 26.0650722 x_1 x_2$} \\[.5ex]
&  {\small
$
- 1.20581071 x_2^2
+ 144.301208 x_1
+ 19.6078596 x_2
- 78.8825355$}\\[1ex]
$P_{0,3}$ & {\small  $x_2^3
+ 48.9437310 x_1^2 x_2
+ 12.1645735 x_1 x_2^2
+ 65.1198605 x_1^3
- 390.719163 x_1^2 $} \\[.5ex]
&  {\small
$
- 195.774924 x_1 x_2
- 24.3291470 x_2^2
+ 781.4383265x_1
+ 195.774924 x_2$}\\
& {\small $- 520.958883$}\\[.5ex]
\hline \\
\end{tabular}
\caption{Entries $P_{n-k,k}$ of the first VOPS with respect to the weight function \eqref{ex_koor} for degree $0 \leqslant n \leqslant 3$ and for $0 \leqslant k \leqslant n$. }
\label{tVOPS_Koor}
\end{table}

\medskip

In this case the matrix in \eqref{k=n} for $n=2$, is given by
$$
\Lambda_2 =
\begin{pmatrix}
0.000144303270 & 0 & 0 \\
-0.00059443234  & 0.000138244373 & 0  \\
0.00247304236    & -0.00113123217 & 0.000759945541
\end{pmatrix}.
$$

\medskip

\subsection{VOPS in a shifted simplex}

We consider the family of VOPS on the shifted simplex associated with  the weight function
\begin{equation}\label{shifted}
W(x_1,x_2)= (x_1-a)^\alpha (x_2-a)^\beta  (a+b-x_1-x_2)^\gamma,
\end{equation}
where $\alpha, \beta, \gamma > -1$, defined on the region
$$\Omega = \{(x_1, x_2) \in \mathbb{R}^2 : x_1\geqslant a, \, x_2 \geqslant a, \, a+b-x_1-x_2 \geqslant 0 \},$$
with $0<a<b<\infty$. Shifted simplex is represented in Figure~\ref{fig:simplex}.
\begin{center}
	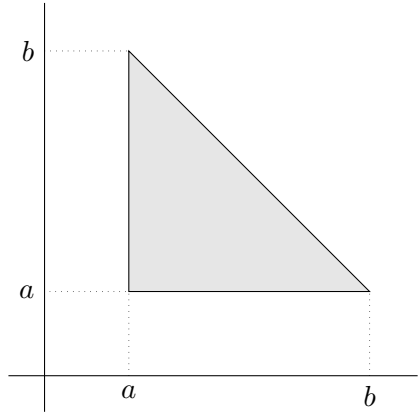
\begin{figure}[h]
\begin{tikzpicture}[xscale=1,yscale=1,scale=1.6]
	\tikzmath{\xa = 0.7; \xb = \xa+2; \yc = \xa; \yd = \xb; }
	\draw (-0.3,0) -- (\xb+0.4,0);
	\draw (0,-0.3) -- (0,\yd+0.4);
	\coordinate [label=below:{$a$}]  (a) at (\xa,0);
	\coordinate [label=below:{$b$}]  (b) at (\xb,0);
	\coordinate [label=left:{$a$}]  (c) at (0,\yc);
	\coordinate [label=left:{$b$}]  (d) at (0,\yd);
	\coordinate (ac) at (\xa,\yc);
	\coordinate (ad) at (\xa,\yd);
	\coordinate (bc) at (\xb,\yc);
	\draw[very thin, dotted] (a) -- (ac) -- (c);
	\draw[very thin, dotted] (d) -- (ad);
	\draw[very thin, dotted] (b) -- (bc);
	\filldraw[fill=black!10] (ac) -- (ad) -- (bc) -- cycle;
	\end{tikzpicture}
	\caption{Shifted simplex.}\label{fig:simplex}
	\end{figure}
\end{center}

Choosing $\alpha=3, \beta=2, \gamma=1, a=1$, and $b=2$, the weight function
$$
W(x_1,x_2)= (x_1-1)^3 (x_2-1)^2  (3-x_1-x_2)
$$
is defined in $\Omega = \{(x_1, x_2) \in \mathbb{R}^2 \, : \, x_1\geqslant 1, \, x_2 \geqslant 1, \, 3-x_1-x_2 \geqslant 0 \}$.

Using Theorem \ref{3.2} we can calculate the monic VOPS associated with  $W(x_1,x_2)$ given by \eqref{shifted}, see Table \ref{tVOPS_shifted}.

\begin{table}[ht!]
\begin{tabular}{ll}
\hline \hline
$P_{0,0}$ & {\small $1$}\\[1ex]
\hline
$P_{1,0}$ & {\small $x_1 -  1.43796769$}\\[1ex]
$P_{0,1}$ & {\small $x_2 -  1.32696025$}\\[.5ex]
\hline
$P_{2,0}$ & {\small $x_1^2
- 2.88799337 x_1
+ 0.00263767 x_2
+ 2.05644459 $}\\[1ex]
$P_{1,1}$ & {\small $x_1x_2
- 1.26099644 x_1
- 1.35144600 x_2
+ 1.71293494$}\\[1ex]
$P_{0,2}$ & {\small $x_2^2
+ 0.00131771 x_1
- 2.70824586 x_2
+ 1.80864659$}\\[.5ex]
\hline
$P_{3,0}$ & {\small $x_1^3
- 4.34366084 x_1^2
+ 0.00992183 x_1 x_2
+ 0.00014521 x_2^2
+ 6.21350091 x_1
$}  \\[.5ex]
&  {\small
$- 0.01365970 x_2
- 2.92734391$}\\[1ex]
$P_{2,1}$ & {\small $x_1^2 x_2
- 1.21564853 x_1^2
- 2.74130572 x_1 x_2
+ 0.00177590 x_2^2
+ 3.34491332 x_1$} \\[.5ex]
&  {\small
$+ 1.85415301 x_2
- 2.27334201$}\\[1ex]
$P_{1,2}$ & {\small  $x_1x_2^2
+ 0.00091564 x_1^2
- 2.58880086 x_1 x_2
- 1.29133163 x_2^2
+ 1.65554061 x_1$} \\[.5ex]
&  {\small
$+ 3.35648957 x_2
- 2.15629266$}\\[1ex]
$P_{0,3}$ & {\small  $x_2^3
+ 0.00006145 x_1^2
+ 0.00572113 x_1 x_2
- 4.11853295 x_2^2
- 0.00732316 x_1$} \\[.5ex]
&  {\small
$+ 5.58718398 x_2
- 2.49732589$}\\[.5ex]
\hline \\
\end{tabular}
\caption{Entries $P_{n-k,k}$ of the first VOPS with respect the weight function \eqref{shifted} for degree $0 \leqslant n \leqslant 3$ and for $0 \leqslant k \leqslant n$. }
\label{tVOPS_shifted}
\end{table}

Also,
$$
\Lambda_2 = \langle \mathbb{X}_{2},\mathbb{P}_2^T\rangle_2 = \begin{pmatrix}
0.000263986510  & 0.000143544887 & 0.000089632056 \\
0.000143544887  & 0.000157172682 & 0.000132333688 \\
0.000089632057  & 0.000132333688 & 0.000214040951
\end{pmatrix}
$$
and
$$
\Upsilon_2 = \langle \mathbb{X}_{-1},\mathbb{P}_2^T\rangle_2 = \begin{pmatrix}
0.000090452349 & 0.000049064200 & 0.000030587667 \\
0.000036048468 & 0.000053277399 & 0.000086301383
\end{pmatrix}.
$$

\bigskip

\section*{Acknowledgements}

C.F.B. thanks the support by the grant 2022/09575-5 from FAPESP, Sao Paulo, Brazil, and the support by the grant 88887.310463/2018-00 from the program CAPES/PrInt of Brazil.

A.M.D., L.F., and T.E.P. thank grant CEX 2020-001105-M funded by MCIN/AEI/10.13039/501100011033 and Research Group Goya-384, Spain.

T.E.P. also thanks grant 88887.911704/2023-00 from CAPES/PrInt of Brazil.


\end{document}